\documentclass[numbers,webpdf,imaiai]{ima-authoring-template}%
\graphicspath{{Fig/}}
\usepackage{subfigure}

\usepackage{tikz}
\usetikzlibrary{positioning}
\theoremstyle{thmstyletwo}
\newcommand{\etal}{et al.}

\newtheorem{theorem}{Theorem}
\newtheorem{proposition}[theorem]{Proposition}   
\newtheorem{lemma}[theorem]{Lemma}               
\newtheorem{remark}{Remark}                      
\numberwithin{equation}{section}
\usepackage{amsthm,amsmath,amssymb}
\usepackage{bm}
\usepackage{mathrsfs} 
\usepackage{subfigure}
\usepackage{booktabs}

\newcommand{\Rmnum}[1]{\expandafAter@slowromancap\romannumeral #1@}
\begin{document}

	\copyrightyear{2021}
	\vol{00}
	\pubyear{2021}
	\access{Advance Access Publication Date: Day Month Year}
	\appnotes{Paper}
	\copyrightstatement{Published by Oxford University Press on behalf of the Institute of Mathematics and its Applications. All rights reserved.}
	\firstpage{1}
	
	
	\title{Observer design and boundary output feedback stabilization  for semilinear parabolic system over general multidimensional domain}

        \author{Kai Liu
	\address{\orgdiv{School of Mathematics and KL-AAGDM}, \orgname{Tianjin University}, \orgaddress{ \postcode{300354}, \state{Tianjin}, \country{China}}}}
\author{Hua-Cheng Zhou \address{\orgdiv{School of Mathematics}, \orgname{Central South University}, \orgaddress{ \postcode{410075}, \state{Changsha}, \country{China}}}}
\author{Zhong-Jie Han\footnote{Corresponding author: Zhong-Jie Han \href{zjhan@tju.edu.cn}{(zjhan@tju.edu.cn)}}
	\address{\orgdiv{School of Mathematics and KL-AAGDM}, \orgname{Tianjin University}, \orgaddress{ \postcode{300354}, \state{Tianjin}, \country{China}}}}
	\author{Xiangyang Peng\address{\orgdiv{School of Intelligence Science and Engineering and the Guangdong Key Laboratory of Intelligent Morphing Mechanisms and Adaptive Robotics}, \orgname{Harbin Institute of Technology (Shenzhen)}, \orgaddress{ \postcode{518055}, \state{Shenzhen}, \country{China}}}}
	\authormark{Kai Liu et al.}

	\received{Date}{0}{Year}
	\revised{Date}{0}{Year}
	\accepted{Date}{0}{Year}
	\abstract{\textbf{Abstract}:
    This paper investigates the output feedback stabilization of parabolic equation with Lipschitz nonlinearity over general multidimensional domain using spectral geometry theories. First, a novel nonlinear observer is designed, and the error system is shown to achieve any prescribed decay rate by leveraging the Berezin-Li-Yau inequality from spectral geometry, which also provides effective guidance for sensor placement. Subsequently, a finite-dimensional state feedback controller is proposed, which ensures the quantitative rapid stabilization of the linear part. By integrating this control law with the observer, an efficient boundary output feedback control strategy is developed. The feasibility of the proposed control design is rigorously verified for arbitrary Lipschitz constants, thereby resolving a persistent theoretical challenge. Finally, a numerical case study confirms the effectiveness of the approach.
    }    
	
	\maketitle
\section{Introduction}
	\subsection{Background}
	Nowadays, ongoing technological requirements have shifted research focus to consider the  nonlinear partial differential equation (PDE) systems, which can be used to properly describe the dynamic behavior of the biological, physical and chemical system \cite{furter1989local,hagen2004distributed,temam2012infinite}. When it comes to the stabilization problem of nonlinear PDE systems, there are some elegant control design schemes, including but not limited to finite-dimensional control design \cite{el2023exponential,hagen2004distributed,karafyllis2021lyapunov,katz2023global,lhachemi2022nonlinear,Lasiecka_Triggiani_2000}, backstepping control design \cite{chowdhury2024local,vazquez2008control} and linear quadratic control design \cite{hagen2003spillover,hagen2004distributed}.  
	
	The objective of this paper is to develop  control design to stabilize  multidimensional nonlinear parabolic equation. There have been some important literature addressing multidimensional parabolic PDE systems \cite{barbu2013boundary,feng2022boundary,lhachemi2025boundary,munteanu2017stabilisation,wang2024delayed,xiang2024}. For the linear multidimensional system, Feng \etal \  \cite{feng2022boundary} developed a dynamic feedback control for stabilizing an unstable heat equation  using dynamic compensation approach. Furthermore, Lhachemi \etal ~ \cite{lhachemi2025boundary}  achieved the exponential stabilization for 2-D and 3-D parabolic equations, focusing on systems governed by linear differential operators that can be diagonalized by weighting functions. For the delayed cases, Wang and Fridman \cite{wang2024delayed} proposed a kind of finite-dimensional observer-based control for 2-D linear parabolic PDEs with time-varying input/output delays, utilizing the Lyapunov functional in combination with Halanay inequality to achieve exponential stability.  For the nonlinear cases, Barbu \cite{barbu2013boundary} designed the linear boundary control strategy for the multidimensional nonlinear parabolic system, demonstrating exponential stabilization for equilibrium solutions under certain conditions and assumptions.  Munteanu \cite{munteanu2017stabilisation} extended Barbu's  work \cite{barbu2013boundary} by removing the linear independence assumption of the eigenfunctions, thus providing a linear state feedback control design applicable to a broader class of nonlinear parabolic equations. However, the explicit control design for globally stabilizing multi-dimensional parabolic systems involving arbitrary Lipschitz nonlinearity remains unresolved.

   On the other hand, a significant challenge in practical implementations arises from the fact that only partial state measurements are typically available, which necessitates the development of effective state reconstruction techniques for nonlinear parabolic PDE systems. Although various observer designs have been proposed for one-dimensional linear or nonlinear parabolic PDE systems \cite{hagen2003spillover,liu2025output,meurer2013extended,smyshlyaev2005backstepping}, as well as for multidimensional linear systems \cite{feng2022boundary,lhachemi2025boundary}, the literature on observer design for multidimensional nonlinear parabolic PDE systems remains limited. 

   To address these challenges and overcome the aforementioned limitations, this paper proposes a novel observer-based  control framework. The proposed method is designed to ensure the stabilization of nonlinear parabolic PDE systems over general multidimensional domain when only partial state measurements are available, thereby offering a more flexible and widely applicable solution.
	
	\subsection{Problem formulation}
Consider the following nonlinear parabolic system:
\begin{equation}\label{old}
\begin{cases}
\partial_t z(x,t) = \Delta z(x,t) + f(z(x,t)), & t>0,\;x\in\Omega, \\
z(x,t) = U(x,t), & t>0,\;x\in\varGamma_1, \\
z(x,t) = 0, & t>0,\;x\in\varGamma_2, \\ 
z(x,0) = z^0(x), & x \in \Omega,
\end{cases}
\end{equation}
where $\Omega \subset \mathbb{R}^d,\;d\in \mathbb{N}^+$ is a smooth bounded domain with boundary $\partial\Omega = \varGamma_1 \cup \varGamma_2$. Here, $\varGamma_1$ and $\varGamma_2$ are disjoint, and $\varGamma_1$ has positive measure. The boundary control $U(x,t) = \sum_{i=1}^N U_i(x,t)$ acts on $\varGamma_1$, where the positive integer $N$ will be specified later. The nonlinear function $f: \mathbb{R} \to \mathbb{R}$ is known and satisfies both global Lipschitz condition with constant $L > 0$ and the equilibrium condition $f(0) = 0$. The initial state $z^0\in L^2(\Omega)$. 

Suppose that the domain $\Omega$ is partitioned into $M$ disjoint Lipschitz subdomains $\{\Omega_i\}_{i=1}^M$ such that $\Omega$ equals the interior of the closure of $\cup_{i=1}^M\Omega_i$ satisfying 
\begin{equation}\label{condition6}
\max_{1 \leq i \leq M} \operatorname{Vol}_d(\Omega_i) < C_d^{-1}\left(\frac{d}{L(d+2)}\right)^{d/2},
\end{equation}
where $\operatorname{Vol}_d(\Omega_i)$ denotes the $d$-dimensional volume (Lebesgue measure) of $\Omega_i$ for $i = 1, \dots, M$, $C_d :=  (4\pi)^{-d/2}[\Gamma(d/2 + 1)]^{-1}$ is Weyl constant and $\Gamma(\cdot)$ is the Gamma function. The measurement $h(x,t)$ of system \eqref{old} is defined as:
	\begin{equation}\label{h1}
	h(x,t) = \sum_{1 \leq i < j \leq M} \chi_{\varGamma_{ij}}(x) z(x,t),
	\end{equation}
	where $\varGamma_{ij} = \Omega_i \cap \Omega_j$ is the interface between adjacent subdomains, $\chi_{\varGamma_{ij}}$ is the indicator function of $\varGamma_{ij}$. 
    
    To facilitate understanding of the domain decomposition, Fig~\ref{fig:domain_decomp} provides a 2-D example where the  domain $\Omega$ is divided into six  subdomains $\Omega_1,\dots,\Omega_6$, with $\varGamma_{ij}\;(i,j=1,2,...,6)$  denoting the interface between adjacent subdomains $\Omega_i$ and $\Omega_j$.
	
	\begin{figure}[htbp]
		\centering
		\includegraphics[width=0.3\textwidth]{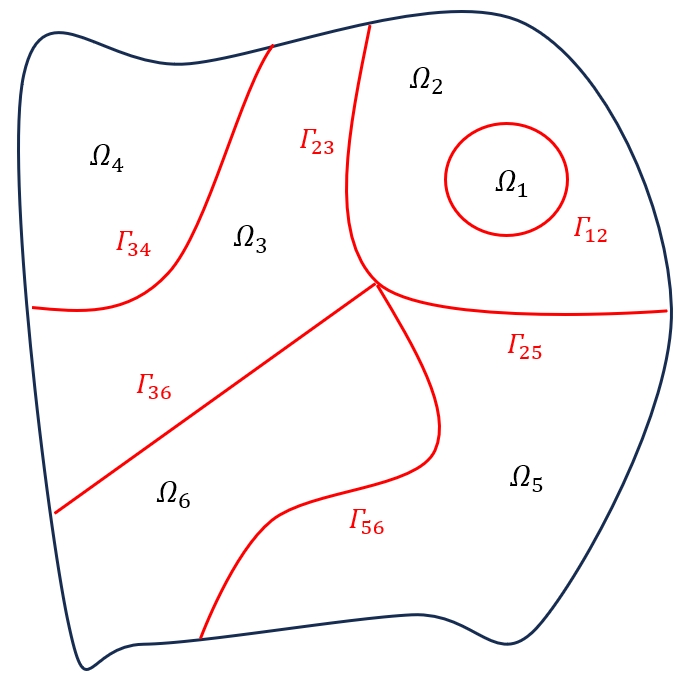}
		\caption{Illustrative example}
		\label{fig:domain_decomp}
	\end{figure}

This paper aims to develop a novel observer design for system \eqref{old} under the measurement  \eqref{h1}, and proposes an effective feedback control law to achieve exponential stabilization of the multidimensional nonlinear parabolic PDE system.
	\subsection{Novelties and Contributions} 
\noindent The novelties and contributions of this paper are listed as follows:
\begin{itemize}
    \item[(1)] We propose a novel nonlinear observer architecture using measurements \eqref{h1} to estimate the state of multidimensional parabolic PDEs in real-time—essential for designing output feedback control strategies. Furthermore, the observer's tracking error is exponentially convergent to $0$. To our knowledge, this work presents the first observer design for such multidimensional nonlinear parabolic PDE systems.
    \item[(2)] We develop a constructive application of spectral geometry theory to the sensor placement problem in multidimensional domains. Specifically, building upon the Berezin--Li--Yau inequality \cite{berezin1972convex,li1983schrodinger}, we establish a sensor placement configuration strategy that guarantees exponential convergence of the observer for general multidimensional domains.
    \item[(3)] We present an observer-based boundary feedback control law for multidimensional nonlinear parabolic PDE systems. In contrast to existing results \cite{katz2022global,katz2023global,lhachemi2022global}, we provide a rigorous proof that for spatial dimensions $d \in \{1, 2, 3\}$, the proposed control strategy guarantees exponential stabilization even for systems with arbitrarily large Lipschitz coefficients.
\end{itemize}
	\subsection{Organization}
    The remainder of this paper is structured as follows. Section \ref{sec2} develops the nonlinear observer and proves tohat its  error system is exponential stable. Section \ref{sec3} designs the boundary output–feedback controller and establishes the stability of the close-loop. Section \ref{sec4} provides numerical simulations to validate the theoretical results. Section \ref{sec5} concludes the paper and discusses possible directions for future research.
	\subsection{Notation} 

Let $\mathbb{R}^{N_1 \times N_2}$ denote the space of all real-valued $N_1 \times N_2$ matrices, equipped with the  2-norm $\|\cdot\|_2$.
For a matrix $\mathbf{A} \in \mathbb{R}^{N \times N}$, we denote its inverse (when it exists) by $\mathbf{A}^{-1}$.  

The space $L^2(\Omega)$ consists of all square-integrable functions $z\colon \Omega \to \mathbb{R}$ endowed with the inner product $\langle f, g \rangle = \int_\Omega f(x)g(x) \, \mathrm{d}x$ and induced norm $\|z\|^2 = \langle z, z \rangle$, while $L^2(\varGamma_1)$ denotes the Hilbert space of functions on the boundary $\varGamma_1$ with inner product $\langle w, v \rangle_{L^2(\varGamma_1)} = \int_{\varGamma_1} w(x)v(x) \, \mathrm{d}\sigma$ and norm $\|w\|_{L^2(\varGamma_1)}^2 = \langle w, w \rangle_{L^2(\varGamma_1)} < \infty$, where $\sigma$ is the Lebesgue measure on $\varGamma_1$.  

For functions $f(N) $ and $g(N) $, $f(N) \lesssim g(N)$  if there exist constants $C > 0 $ and $N_0 > 0 $ such that $|f(N)| \leq C|g(N)| $ for all $N > N_0 $. The notation $f(N) = o(g(N))$ indicates that $\lim_{N \to \infty} f(N)/g(N) = 0$. 

	\section{State Estimation}\label{sec2}
    \subsection{Preliminaries}
    	Let us consider the Sturm-Liouville operator $\mathscr{A}$:
	\begin{equation*}
	\begin{cases}
	\mathscr{A}y=-\Delta y, \; y \in \mathscr{D}(\mathscr{A}),\\
	\mathscr{D}(\mathscr{A})=\{y\in H^2(\Omega)\cap H_0^1(\Omega)|\; y=0\;\mbox{on}\; \varGamma_2 \},
	\end{cases}
	\end{equation*}
	where $ H_0^1(\Omega)=\{y\in H^1(\Omega)| y=0\; \mbox{on}\; \varGamma_1\}$. 
    Note that $\mathscr{A}$ is a self-adjoint operator with a compact resolvent, which guarantees the existence of a sequence of real eigenvalues $\{\lambda_n\}_{n=1}^\infty$. These eigenvalues, denoted by $\{\lambda_n\}_{n=1}^\infty$, satisfy
	\begin{equation*}\label{eigen}
	0 < \lambda_1 < \lambda_2 \leq \lambda_3 \leq ... \leq \lambda_n \leq ... \to +\infty,
	\end{equation*}
	with corresponding eigenfunctions $\{\phi_n\}_{n=1}^\infty$ determined by
	\begin{equation*}\label{eigenfunction}
	\begin{cases}
	\Delta \phi_n = -\lambda_n\phi_n, \quad \|\phi_n\| = 1,\\
	\phi_n = 0, \quad x\in \partial\Omega.
	\end{cases}
	\end{equation*}
    All eigenvalues have finite algebraic multiplicity, and the first eigenvalue $\lambda_1$ is simple \cite{evans2010partial}. The spectral properties of the operator $\mathscr{A}$ are characterized by the following well-known results
\begin{proposition}\label{pro-1}
The eigenvalue-eigenfunction pairs $(\lambda_k, \phi_k)_{k=1}^{\infty}$ possess the following spectral properties:\\
    \textbf{(1)} Weyl's Law \cite{weyl1912asymptotische}: The eigenvalues of the operator $\mathscr{A}$ satisfy the asymptotic behavior
    $$
    \lambda_k \sim \left(\frac{k}{C_d \mathrm{Vol}_d(\Omega)}\right)^{\frac{2}{d}},
    $$
    where $\sim$ denotes asymptotic equivalence as $k \to \infty$.\\
    \textbf{(2)} Berezin-Li-Yau Inequality\cite{berezin1972convex,li1983schrodinger}: The $k$-th  eigenvalue of the operator $\mathscr{A}$ satisfies the lower bound
    $$
    \lambda_k \geq \frac{d}{d+2} \left( \frac{k}{C_d \mathrm{Vol}_d(\Omega)} \right)^{\frac{2}{d}}.
    $$\\
    \textbf{(3)} Rellich Inequality \cite{hassell2002upper}: There exist constants $c, C > 0$ independent of $k$ such that $$c \lambda_k \leq \|\partial_{\bm{n}}\phi_k\|_{L^2(\varGamma_1)}^2 \le C \lambda_k,$$
    where $\bm{n}$ is the unit outward normal vector to $\varGamma_1$
\end{proposition}

Following \cite{munteanu2017stabilisation}, we employ a lifting technique to handle boundary control. Define the Dirichlet maps $\mathscr{D}_i \in \mathcal{L}(L^2(\varGamma_1), H^{\frac{1}{2}}(\Omega))$ for $i = 1, \dots, N$ by $\mathscr{D}_i \alpha = w_i$, where $w_i$ is the unique solution to:
\begin{equation}\label{dirichlet}
\begin{cases}
\Delta w_i +2\sum\limits_{i=1}^N\lambda_i\langle \omega_i,\phi_i \rangle\phi_i - k_i w_i = 0 & \text{in } \Omega, \\
w_i = \alpha & \text{on } \varGamma_1,\\
w_i = 0 & \text{on } \varGamma_2.
\end{cases}
\end{equation}
Here, the parameters $k_i$ are chosen as 
\begin{equation}\label{ki}
    k_i = \lambda_i - \|\partial_{\bm{n}} \phi_i\|_{L^2(\varGamma_1)}\lambda_N^{-3/4}, \;\; i = 1, 2, \dots, N.  
\end{equation}
We now proceed to transform the original boundary control system into an equivalent distributed control system with homogeneous boundary conditions. To this end, we introduce the following change of variables:
\begin{equation}\label{dynamic-extension}
p(x,t) = z(x,t) - \sum_{i=1}^{N} \mathscr{D}_i U_i.
\end{equation}
By substituting \eqref{dynamic-extension} into \eqref{old}, we obtain
\begin{equation}\label{new} 
\begin{cases}
\partial_t p(x,t) = \Delta p(x,t) + f\left(p + \sum\limits_{i=1}^{N} \mathscr{D}_i U_i\right) - \sum\limits_{i=1}^{N} \big[\mathscr{D}_i (\dot{U}_i-k_iU_i)-2\lambda_i\langle \mathscr{D}_iU_i,\phi_i\rangle\phi_i\big],\\
\hspace{11.2cm}x\in\Omega, \\
p(x,t) = 0, \; x\in\partial\Omega, \\
p(x,0) = p^0(x), \; x \in \Omega.
\end{cases}
\end{equation}
 \subsection{Observer Design}     
Utilizing these spectral properties, we can explicitly quantify the relationship between the geometry of the domain and the decay rate of the estimation error. This allows us to design a state observer by properly configuring the measurement domains. Accordingly, we propose the following observer for system \eqref{new}:
\begin{equation}\label{observer-1}
\begin{cases}
\partial_t\hat{p}(x,t) = \Delta \hat{p}(x,t) + f\left(\hat{p}+ \sum\limits_{i=1}^{N} \mathscr{D}_i U_i\right) - \sum\limits_{i=1}^{N} \big[\mathscr{D}_i (\dot{U}_i-k_iU_i)-2\lambda_i\langle \mathscr{D}_iU_i,\phi_i\rangle\phi_i\big],\\
\hspace{11.2cm}x\in\Omega, \\
\hat{p}(x,t) = 0,\;  x\in\partial \Omega, \\
\hat{p}(x,t) = z(x,t)-\sum\limits_{i=1}^{N} \mathscr{D}_i U_i, \;x \in \bigcup\limits_{1 \le i < j  \le M} \varGamma_{ij}, \\
\hat{p}(x,0) = \hat{p}^0(x), \; x\in\Omega,
\end{cases}
\end{equation}
where  $\hat{p}^0(\cdot)\in L^2(\Omega)$ is arbitrarily given. The observer error is defined by $$\epsilon(x,t) = p(x,t) - \hat{p}(x,t),$$ which satisfies:
	\begin{equation}\label{error-1}
	\begin{cases}
	\partial_t \epsilon(x,t) = \Delta \epsilon(x,t) + f_p, & t>0,\;x\in\Omega, \\
	\epsilon(x,t) = 0, & t>0,\;x\in\varGamma, \\ 
	\epsilon(x,t) = 0, & t>0,\;x\in\bigcup\limits_{1 \leq i < j \leq M} \varGamma_{ij}, \\
	\epsilon(x,0) = p^0(x)-\hat{p}^0(x),&x\in\Omega,
	\end{cases}
	\end{equation}
	where 
	$f_p = f\big(\hat{p} + \sum_{i=1}^{N} \mathscr{D}_i U_i + \epsilon\big) - f\big(\hat{p} + \sum_{i=1}^{N} \mathscr{D}_i U_i\big)$
	remains Lipschitz continuous with respect to $\epsilon(x,t)$. To precisely characterize the error dynamics, we define piecewise error functions over the domain decomposition:
	\begin{equation*}
	\epsilon(x,t) = 
	\begin{cases}
	\epsilon_i(x,t), & t>0,\;x\in\Omega_i, \quad i = 1,\dots,M, \\
	0, & t>0,\;x\in\varGamma_{ij}, \quad 1 \leq i < j \leq M,
	\end{cases}
	\end{equation*}
where  $\{\epsilon_i\}_{i=1}^M$ satisfies the following PDE systems:
	\begin{equation}\label{e12-1}
	\begin{cases}
	\partial_t \epsilon_i(x,t) = \Delta \epsilon_i(x,t) + f_p, & t>0,\;x\in\Omega_i, \\
	\epsilon_i(x,t) = 0, & t>0,\;x\in\partial\Omega_i,\\
     \epsilon_i(x,0)=p^0(x)-\hat{p}^0(x),&x\in \Omega_i.
	\end{cases}
	\end{equation}
	The existence and uniqueness of classical solutions to the error system \eqref{e12-1} can be established by applying \cite[Theorem 6.3.3]{pazy2012semigroups}.
	 Consider the Sturm-Liouville operator 
	\begin{equation*}
	\begin{split}
	&\mathscr{A}_i:\mathscr{D}(\mathscr{A})\to L^2(\Omega_i),\quad i=1,2,\ldots,M,\\
	&\mathscr{A}_i h = -\Delta h,
	\end{split}
	\end{equation*}
	defined on the domain $\mathscr{D}(\mathscr{A}_i) = \{h(x)\in H^2(\Omega_i) \mid h(x)=0 \text{ for } x\in \partial\Omega_i\}$. Analogous to \eqref{eigenfunction}, each operator $\mathscr{A}$ ($i = 1, 2,\ldots,K$) possesses a  set of eigenvalues 
	\begin{equation*}
	0 < \lambda_{i, 1} < \lambda_{i, 2} \leq \lambda_{i, 3} \leq \cdots \leq \lambda_{i, n} \leq \cdots \to +\infty,
	\end{equation*}
	with corresponding eigenfunctions $\phi_{i,n}$ ($n=1,2,\ldots$) that satisfy the equation:
	\begin{equation*}
	\begin{cases}
	\Delta \phi_{i, n} = -\lambda_{i, n} \phi_{i, n}, & x \in \Omega_i, \\ 
	\phi_{i, n} = 0, & x \in \partial \Omega_i.
	\end{cases}
	\end{equation*}
    
 We have the following result on the stability of the error system \eqref{error-1}.

\begin{theorem}\label{errorthm}
Consider the error system \eqref{error-1} with the system measurement \eqref{h1} satisfying \eqref{condition6}. There exist positive constants $\kappa_1$ and $\kappa_2$ such that the exponential stability estimate holds:
	\begin{align}\label{expon-est}
	\|\epsilon(\cdot,t)\| \leq \kappa_1e^{-\kappa_2 t}\|p^0(\cdot)-\hat{p}^0(\cdot)\|.
	\end{align}
\end{theorem}
	\begin{proof} 
		 Let $T_i(t),\;i=1,2,\dots,M$ be a $C_0$ semigroup and let $-\mathscr{A}_i$ be its infinitesimal generator, which means 
		\begin{equation}\label{suanzi}
		\|T_i(t)\|_{\mathscr{L}(L^2(\Omega_i))}\leq e^{-\lambda_{i,1}t}.
		\end{equation}
		Then, solving \eqref{e12-1} yields
		$$
		\epsilon_i(t)=T_i(t)(p^0-\hat{p}^0)+\int_{0}^{t}T_i(t-\tau)f_p(\tau)d\tau.
		$$
		Thus, we immediately obtain
		\begin{equation*}
		\begin{split}
		\|\epsilon_i\|&\leq e^{-\lambda_{i,1}t}(\|z^0-\hat{z}^0\|+\int_{0}^{t}\|T_i(-\tau)\|\|f_z(\tau)\|d\tau)\\
		&\leq e^{-\lambda_{i,1}t}(\|z^0-\hat{z}^0\|+L\int_{0}^{t}\|T_i(-\tau)\|\|\epsilon_i(\tau)\|d\tau),
		\end{split}
		\end{equation*}
		where the Lipschitz property of $f_p(\cdot)$ has been used. Then, it can be deduced by the Gronwall's inequality that
		\begin{equation}\label{est-3}
		\|\epsilon_i\|\leq \|p^0-\hat{p}^0\|e^{(-\lambda_{i,1}+L)t},\;i=1,2,...,M.
		\end{equation}
		Proposition \ref{pro-1} implies that 
        \begin{equation*} 
        \begin{split} 
        \max_{1 \leq i \leq M}\{-\lambda_{i,1}\} = -\min_{1 \leq i \leq M}\{\lambda_{i,1}\} \leq - 2\pi\bigg(\max_{1 \leq i \leq M} \operatorname{Vol}(\Omega_i)\bigg)^{-1}, 
        \end{split} 
        \end{equation*} which, combined with \eqref{condition6}, leads to 
        \begin{equation}\label{test} 
        L + \max_{1 \leq i \leq M}\{-\lambda_{i,1}\} < 0. 
        \end{equation}
        Recall \eqref{est-3}, we can conclude that $\epsilon_i,\;i=1,2,...,M$ is exponentially stable.
		  The proof is completed.\end{proof}
   \begin{remark}
    Invoking the Berezin-Li-Yau inequality, we demonstrate that the measurement can be chosen such that $L+\max_{1 \leq i \leq M}\{-\lambda_{i,1}\}$ becomes arbitrarily negative number. This implies that the error system \eqref{e12-1} realizes quantitative and rapid stabilization. Further details regarding quantitative rapid stabilization can be found in \cite{xiang2024}.
    \end{remark}
    \subsection{Sensor Placement Strategy on Rectangular Domains}\label{sec2.3}
Consider a nonlinear parabolic system defined on the rectangular domain $\Omega = (0, 1) \times (0, 1)$ with  $\varGamma_1 = \{0\} \times (0, 1)$ representing the controlled boundary. The system dynamics are governed by the parabolic PDE with associated boundary and initial conditions.
The measurement of the above system  is given by:
\begin{equation}\label{measurement}
h(x,t) =  \sum_{i=1}^{M_1}\sum_{j=1}^{M_2} \left(\chi_{\varGamma_{i,j}^H} + \chi_{\varGamma_{i,j}^V}\right) z(x,t),
\end{equation}
in which   (see figure \ref{fig:domain_decomp1}) consist of horizontal and vertical line segments defined as:
\begin{equation*}
    \begin{split}
        &\varGamma_{i,j}^H = (b_{i-1},b_i)\times\{a_j\}, \quad i=1,\dots,M_1+1,\\
        &\varGamma_{i,j}^V = \{b_i\}\times(a_{j-1},a_j), \quad j=1,\dots,M_2+1,
    \end{split}
\end{equation*}
with partition points $0=a_0 < a_1 < \cdots < a_{M_1} < a_{M_1+1}=1$ and $0=b_0 < b_1 < \cdots < b_{M_2} < b_{M_2+1}=1$. This partition induces a decomposition of $\Omega$ into rectangular subdomains 
\begin{equation*}
    \Omega_{i,j} = (b_{i-1}, b_i) \times (a_{j-1}, a_j), \;i=1,\dots,M_1+1,\;j=1,\dots,M_2+1.
\end{equation*}
The proposed observer system is given as in \eqref{observer-1}. 
\begin{figure}[htbp] 
    \centering
    \includegraphics[width=0.6\textwidth]{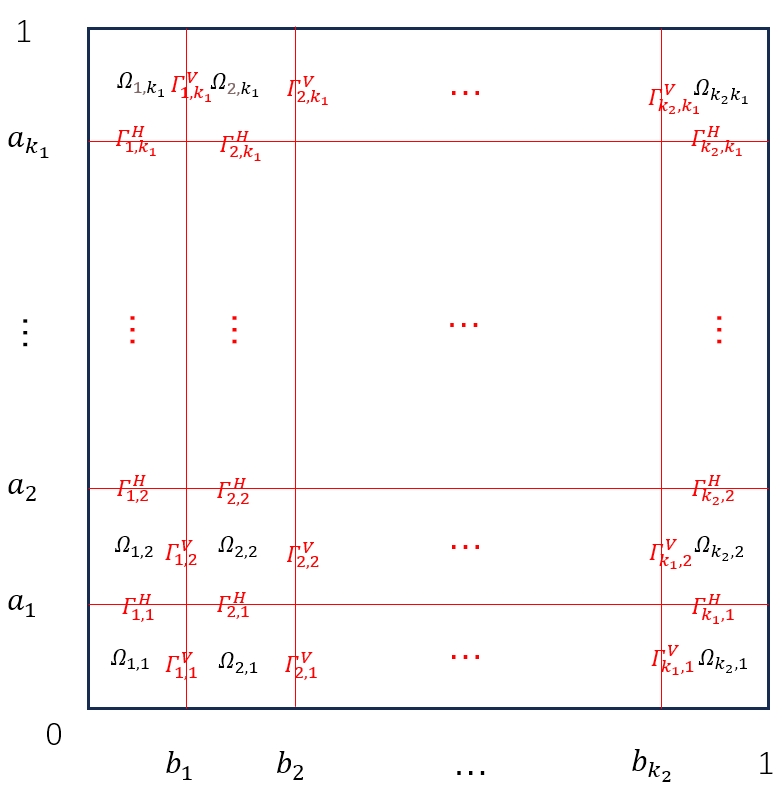}
    \caption{Domain decomposition}
    \label{fig:domain_decomp1}
\end{figure}
The estimation errors $\epsilon(x,y,t) = p(x,y,t) - \hat{p}(x,y,t)$ exhibit the following piecewise behavior:
\begin{equation*}\label{error_z}
\epsilon(x,y,t) = 
\begin{cases} 
\epsilon_{ij}(x,y,t), & (t,x,y)\in (0,\infty)\times\Omega_{i.j}, \\
0, & (t,x,y)\in (0,\infty)\times\bigcup_{i,j}(\varGamma_{i,j}^H \cup \varGamma_{i,j}^V).
\end{cases}
\end{equation*}
The error components $\epsilon_{ij}(x,y,t)$  satisfy the following parabolic PDE systems:
             \begin{align*}
			\begin{cases}
			\partial_t \epsilon_{ij}(x,y,t) = \Delta \epsilon_{ij}(x,y,t) + f_p(\epsilon_{ij}(x,y,t)), & (x,y) \in \Omega_{ij}, \\
			\epsilon_{ij}(x,y,t) = 0, & (x,y) \in   \partial\Omega_{ij}.
			\end{cases}
			\end{align*}
		Consider  the operator $\mathscr{A}_{ij}$, which is defined as
		\begin{equation*}
		\begin{cases}
		\mathscr{A}_{ij}y=-\Delta y, \;\forall y \in \mathscr{D}(\mathscr{A_{ij}}),\\
		\mathscr{D}(\mathscr{A_{ij}})=\{y\in H^2(\Omega_{ij})|\ y=0\;\mbox{on}\; \partial\Omega_{ij}  \}.
		\end{cases}
		\end{equation*}
		By direct computation, the principal eigenvalues of the operator $\mathscr{A}_{ij}$ are given by $$\lambda_{ij,1}=\pi^2 [ (b_i-b_{i-1})^{-2}+(a_j-a_{j-1})^{-2} ].$$
		To ensure \eqref{test} holds, it suffices to choose suitable partition  to guarantee
		\begin{equation}\label{condition7}
		L< \pi^2\bigg(\min_{1\leq i \leq M_1}\{(b_i-b_{i-1})^{-2}\}+\min_{1\leq j \leq M_2}\big\{(a_j-a_{j-1})^{-2}\big\}\bigg).
		\end{equation}
This inequality reveals that for any prescribed Lipschitz constant $L$, \eqref{condition7} can be satisfied by selecting sufficiently large integers $M_1$ or $M_2$ in \eqref{measurement}. 

 To achieve efficient sensor placement, we further  characterize the partition strategy. For any given  $L$, to minimize $M_1$ and $M_2$, we need to find  partition strategies to maximize $\min_{1\leq i\leq M_1+1}(b_i-b_{i-1})^{-2}$ and $\min_{1\leq j\leq M_2+1}(a_j-a_{j-1})^{-2}$ for the horizontal and vertical divisions, respectively.
 Applying Jensen's inequality, we obtain
\begin{align*}
 \min_{1\leq i \leq M_1} \frac{1}{(b_i - b_{i-1})^2} 
    &\leq \frac{1}{M_1+1}\sum_{i=1}^{M_1+1} (b_i - b_{i-1})^{-2} \leq (M_1+1)^2,\\
    \min_{1\leq j \leq M_2} \frac{1}{(a_j - a_{j-1})^2}
    &\leq \frac{1}{M_2+1}\sum_{j=1}^{M_2+1} (a_j - a_{j-1})^{-2} \leq (M_2+1)^2.\\
\end{align*}
The equality is achieved if and only if the partitions are equidistant, specifically:
\begin{itemize}
\item For $M_1$ vertical divisions: $a_j = \frac{j}{M_1+1}$; 
\item For $M_2$ horizontal divisions: $b_i = \frac{i}{M_2+1}$. 
\end{itemize}
It follows the preceding analysis, for any given $L$, by (\ref{condition7}),  we have $$L < \pi^2(M_1 + 1)^2 + \pi^2(M_2 + 1)^2,$$
under equidistant partition. 
To minimizes the total number of lines equipped with sensors $M$ ($M=M_1+M_2$) satisfying the above inequality, we need to maximize $(M_1 + 1)^2 + (M_2 + 1)^2$. It can be calculated from Lagrange multiplier method ‌that the maximum value is achieved when either all sensors are placed vertically: $(M_1,M_2) = (M,0)$ or all sensors are placed horizontally: $(M_1,M_2) = (0,M)$.
\section{Controller Design}\label{sec3}
This section develops the boundary  control law in three steps. Section~\ref{sec3.1} introduces the preliminary knowledge. Based on Section~\ref{sec3.1}, Section~\ref{sec3.2} constructs a state feedback control law assuming full state availability. Finally, Section~\ref{sec3.3} integrates this control law with the observer designed in Section ~\ref{sec2} to obtain the output feedback control.
    \subsection{Preliminaries} \label{sec3.1}
Hereafter, for any function $h(\cdot) \in L^2(\Omega)$, we denote $h_n = \langle h(\cdot), \phi_n(\cdot) \rangle$. The solution to the first equation of system \eqref{new} can be expanded in the eigenfunction basis as:
\begin{equation*}
p(x,t) = \sum_{n=1}^{\infty} p_n(t) \phi_n(x), \quad \text{where} \quad p_n(t) := \langle p(\cdot,t), \phi_n(\cdot) \rangle.
\end{equation*}
Differentiating both sides of $p_n(t) = \langle p(\cdot,t), \;\phi_n(\cdot) \rangle,\;n \in \mathbb{N}^+$ with respect to time $t$ yields
\begin{equation}\label{infinite}
\dot{p}_n = -\lambda_n p_n + f_n(z) - \sum_{i=1}^{N} \langle \mathscr{D}_i (\dot{U}_i-k_iU_i), \phi_n \rangle-2\sum\limits_{i=1}^N\lambda_i \langle \mathscr{D}_iU_i,\phi_i\rangle\delta_{in}.
\end{equation}
An application of Green's formula yields
\begin{equation}\label{ys}
\langle \mathscr{D}_iU_i, \phi_n \rangle = -\frac{\big\langle U_i, \partial_{\bm{n}} \phi_n \big\rangle_{L^2(\varGamma_1)}}{k_i + \lambda_n-2\lambda_n\sum_{l=1}^N\delta_{ln}}.
\end{equation}
We define
\begin{equation}\label{Ui}
    U_i = \langle \bm{u}, \bm{T}_i \rangle_{\mathbb{R}^N},
\end{equation}
where $\bm{u}$ is given by \eqref{U} (state feedback) or \eqref{observer-based control} (output feedback), and $\bm{T}_i = [ -\frac{\partial_{\bm{n}} \phi_1}{k_i-\lambda_1}, \dots, -\frac{\partial_{\bm{n}} \phi_N}{k_i-\lambda_N} ]^{\top}$.
Let $\bm{p}_s = [p_1, \dots, p_N]^{\top}$, $\bm{f}_s = [f_1, \dots, f_N]^{\top}$ and $\bm{A}_s = \operatorname{diag}\{-\lambda_1, \dots, -\lambda_N\}$. Direct computation leads to
\begin{equation}\label{ps}
    \dot{\bm{p}}_s = \bm{A}_s\bm{p}_s + \bm{f}_s - \bm{B}\dot{\bm{u}} - \bm{C}\bm{u},
\end{equation}
where the entries of $\bm{B}, \bm{C} \in \mathbb{R}^{N \times N}$ are given by
\begin{equation}\label{a}
\bm{B}_{jn} \!=\! \sum_{i=1}^N\frac{\langle \partial_{\bm{n}} \phi_j,\partial_{\bm{n}} \phi_n \rangle}{(k_i\!-\!\lambda_j)(k_i\!-\!\lambda_n)}, ~ \bm{C}_{jn}\!=\!\sum_{i=1}^N\frac{k_i\langle \partial_{\bm{n}} \phi_j,\partial_{\bm{n}} \phi_n \rangle}{(k_i\!-\!\lambda_j)(k_i\!-\!\lambda_n)}.
\end{equation}
Analogous to \cite[Lemma A.1]{lhachemi2025boundary}, the matrix $m\bm{B}-\bm{C}$ is invertible.
\subsection{State Feedback Control} \label{sec3.2}
We now focus on stabilizing the nonlinear system \eqref{old}. To this end, we propose a linear state feedback control law for the finite-dimensional subsystem \eqref{ps}. The control vector $\bm{u}$ is defined by:
\begin{equation}\label{U}
    \bm{u} = \bm{K}\bm{p}_s,
\end{equation}
with the gain matrix $\bm{K} = -(m\bm{B}-\bm{C})^{-1}(m\bm{I}+\bm{A}_s)$. Here, $m$ is a tuning parameter, the specific selection of which will be determined in the subsequent stability analysis. 
Substituting the control law \eqref{U} into \eqref{ps} yields 
\begin{equation}\label{lin-p}
\begin{cases}
\dot{\bm{p}}_{s} = -m \bm{p}_{s} - (\bm{I} + \bm{B} \bm{K})^{-1} \bm{f}_s,\\
\dot p_n
= -\lambda_n p_n + f_n \\
\hspace{0.9cm}+\sum_{i=1}^N \langle \mathscr{D}_i \big(\langle (k_i+m)\bm{K}\bm{p}_s +\! \bm{K}(\bm{I}+\bm{B}\bm{K})^{-1}\bm{f}_s, \bm{T}_i\big\rangle_{\mathbb{R}^N} \big), \phi_n\rangle,\; n\!>\!N. 
\end{cases}
\end{equation}
Based on \eqref{lin-p}, it allows us to reformulate the closed-loop system as an abstract evolution equation. The resulting dynamics are governed by:
\begin{equation}\label{evolution}
\frac{dp}{dt}= \tilde{\mathcal{A}}p + \mathcal{G}(p),
\end{equation}
where the linear operator $\tilde{\mathcal{A}}$ and the nonlinear operator $\mathcal{G}$ are defined as: 
\begin{equation}\label{ag}
    \begin{split}
        &\tilde{\mathcal{A}} p = \Delta p+ \sum_{i=1}^{N} [\mathscr{D}_i ( (k_i + m) \bm{T}_i^\top\bm{K}\bm{p}_s )-2\lambda_i\langle \mathscr{D}_i\bm{T}_i^\top\bm{K}\bm{p}_s,\phi_i\rangle\phi_i],\\
        &\mathcal{G}(p) = f(z) - \sum_{i=1}^{N} \mathscr{D}_i (\bm{T}_i^\top \bm{K} (\bm{I}+\bm{B} \bm{K})^{-1} \bm{f}_s).
    \end{split}
\end{equation}
Then, we have the following Theorem.
\begin{theorem}\label{stability-thm}
Consider the linear system \eqref{lin-p} with $f \equiv 0$ under the control law $U_i$ defined in \eqref{Ui}. For any prescribed tuning parameter $m > 1/2$, if  $N \in \mathbb{N}^+$ is chosen sufficiently large such that the following stability condition is satisfied:
\begin{equation}\label{condition-1}
m < \lambda_{N+1} - \frac{\|\bm{K}\|_2^2}{2} \sum_{i,j=1}^N \frac{\zeta_{i,j}(m+k_i)^2}{(k_i-\lambda_j)^2},
\end{equation}
where $\zeta_{i,j} = \sum_{n=N+1}^{\infty} \langle \mathscr{D}_i(\partial_{\bm{n}}\phi_j), \phi_n \rangle^2$, then for any initial condition $p^0 \in L^2(\Omega)$, the state satisfies the exponential decay estimate:
\begin{equation}\label{decay_est}
\|p(\cdot,t)\|^2 \leq e^{-(2m-1) t} \|p^0\|^2, \quad \forall t \ge 0.
\end{equation}
In particular, for spatial dimensions $d\in\{1,2,3\}$, condition \eqref{condition-1} can always be satisfied.
\end{theorem}
\begin{proof}
Consider the Lyapunov function
  \begin{equation*} 
  V(t) = \|p\|^2 = \bm{p}_s^{\top}\bm{p}_s + \sum_{n=N+1}^{\infty} p_n^2. 
  \end{equation*}
Differentiating $V(t)$ along  \eqref{lin-p} yields
    \begin{equation*}
        \begin{split}
            \dot{V}(t) = & -2m\bm{p}_s^{\top}\bm{p}_s+ \sum_{n=N+1}^{\infty} \bigg[ -2\lambda_n p_n^2 + 2 \bigg\langle \sum\limits_{i=1}^N \mathscr{D}_i \langle(m+k_i)\bm{K}\bm{p}_s, \bm{T}_i\rangle_{\mathbb{R}^N} , p_n\phi_n \bigg\rangle \bigg].
        \end{split}
    \end{equation*}
Applying Young's inequality, we  obtain
    \begin{equation*}
        \begin{aligned}
            & 2 \bigg\langle \sum\limits_{i=1}^N \mathscr{D}_i \langle(m+k_i)\bm{K}\bm{p}_s, \bm{T}_i\rangle_{\mathbb{R}^N} , \sum_{n=N+1}^{\infty}p_n\phi_n \bigg\rangle \\
            \leq&   \bm{p}_s^{\top}\bm{p}_s+ \|\bm{K}\|_2^2 \sum_{i,j=1}^N \frac{\zeta_{i,j}(m+k_i)^2}{(k_i-\lambda_j)^2}\sum_{n=N+1}^{\infty}p_n^2.
        \end{aligned}
    \end{equation*}
Finally, we arrive at
\begin{equation*}
\begin{split}
\dot{V}(t)\leq&  (-2m + 1 )\bm{p}_s^{\top}\bm{p}_s +  \bigg[ -2\lambda_{N+1}+  \|\bm{K}\|_2^2\sum_{i,j=1}^N \frac{\zeta_{i,j}(m+k_i)^2}{(k_i-\lambda_j)^2}   \bigg] \sum_{n=N+1}^{\infty}p_n^2.
\end{split}
\end{equation*}
With $m$ satisfying \eqref{condition-1}, it follows from Gronwall's inequality that $$V(t) \leq e^{-(2m-1)t}V(0).$$ This, together with the definition of $V(t)$, directly establishes \eqref{decay_est}.

We now show that \eqref{condition-1} always holds for sufficiently large $N$.
 This necessitates estimating the asymptotic orders of the terms $\|\bm{K}\|_2$, $\sum_{i,j=1}^N\|\partial_{\bm{n}}\phi_j\|_{L^2(\varGamma_1)}^2(k_i-\lambda_j)^{-2}$ and $\zeta_{i,j}\cdot\|\partial_{\bm{n}}\phi_j\|_{L^2(\varGamma_1)}^{-2}$.
 
 Firstly, we estimate  $\|\bm{K}\|_2$. Consider the $\bm{K}^{-1}$ defined by the elements:
    \begin{equation}\label{est-5}
        (\bm{K}^{-1})_{jn} = \frac{-1 }{m  -\lambda_j }\sum_{i=1}^N\frac{(m-k_i)\langle \partial_{\bm{n}} \phi_j,\partial_{\bm{n}} \phi_n \rangle}{(k_i-\lambda_j)(k_i-\lambda_n)}.
    \end{equation}
It can be analyzed that $\left| (\bm{K}^{-1})_{jj} \right| > \sum_{n \neq j}^N \left| (\bm{K}^{-1})_{jn} \right|$,
  which means $\bm{K}^{-1}$ is the diagonally dominant matrix.  According to the conclusion presented in \cite[Theorem 2]{qi1984some}, we arrive at
\begin{equation}\label{est-7}
\begin{split}
      \|\bm{K}^{-1}\|_2 &\leq \max\limits_{1\leq j \leq N} \left\{ \bigg|\frac{m  - k_j }{m-\lambda_j }\bigg|\right\}\lambda_N^{3/2}(1+o(1)),\\ 
      \sigma_{\min}(\bm{K}^{-1}) &\geq \min\limits_{1\leq j \leq N} \left\{ \bigg|\frac{m  - k_j }{m  -\lambda_j }\bigg|\right\}\lambda_N^{3/2}(1+o(1)),
\end{split}
\end{equation}
where  $\sigma_{\min}(\bm{K}^{-1})$ denotes the minimum singular value of the matrix $\bm{K}^{-1}$.
    This leads to the estimate:
    \begin{equation}\label{est-6}
        \|\bm{K}\|_2  \!\leq\! \min\limits_{1\leq j \leq N} \left\{ \bigg|\frac{m  \!-\! k_j }{m  \!-\!\lambda_j }\bigg|\right\}^{-1} \lambda_N^{-\frac{3}{2}}(1\!+\!o(1)).
    \end{equation} 
    
Then, we turn our attention to the term $\sum\limits_{i,j=1}^N\|\partial_{\bm{n}}\phi_j\|_{L^2(\varGamma_1)}^2\cdot(k_i-\lambda_j)^{-2}$.  We decompose the summation as follows:
\begin{equation*}
\begin{split}
    \sum_{i,j=1}^N \frac{\|\partial_{\bm{n}}\phi_j\|_{L^2(\varGamma_1)}^2}{(k_i-\lambda_j)^2}
     =N \lambda_N^{\frac{3}{2}} + \sum_{i\neq j}^N \frac{\|\partial_{\bm{n}}\phi_j\|_{L^2(\varGamma_1)}^2}{(k_i-\lambda_j)^2} .
\end{split} 
\end{equation*}
This leads to the  estimate:
\begin{equation}\label{est-11}
\begin{split}
       \sum_{i,j=1}^N \frac{\|\partial_{\bm{n}}\phi_j\|_{L^2(\varGamma_1)}^2}{(k_i-\lambda_j)^2} \lesssim  N \lambda_N^{\frac{3}{2}}.
\end{split}
\end{equation}

Finally, we need to estimate  $\zeta_{i,j}\cdot\|\partial_{\bm{n}}\phi_j\|_{L^2(\varGamma_1)}^{-2}$. By a direct calculation, we obtain:
\begin{equation*}
\begin{split} \zeta_{i,j}\cdot\|\partial_{\bm{n}}\phi_j\|_{L^2(\varGamma_1)}^{-2} 
    = &\sum_{n=N+1}^{\infty} \frac{ \langle \partial_{\bm{n}} \phi_j, \partial_{\bm{n}}\phi_n \rangle^2}{(k_i+\lambda_n)^2\cdot\|\partial_{\bm{n}}\phi_j\|_{L^2(\varGamma_1)}^2}\\
     \leq& \sum_{n=N+1}^{\infty} \frac{\|\partial_{\bm{n}}\phi_n\|_{L^2(\varGamma_1)}^2}{(k_i+\lambda_n)^2}.
\end{split}
\end{equation*}
For the case $d=1$, Proposition \ref{pro-1} ensures that the series $\sum_{n=1}^{\infty}\|\partial_{\bm{n}}\phi_n\|_{L^2(\varGamma_1)}^2(k_i+\lambda_n)^{-2}$ is bounded, which implies that the tail sum $\sum_{n=N+1}^{\infty}\|\partial_{\bm{n}}\phi_n\|_{L^2(\varGamma_1)}^2(k_i+\lambda_n)^{-2}$ vanishes as $N\to +\infty$. Consequently, for $d\in\{1,2,3\}$, applying the Stolz-Ces\`aro theorem \cite{kaczor2000problems}  yields
\begin{equation*}
\begin{split}
&\lim_{N\to\infty}\frac{\sum_{n=N+1}^{\infty}\|\partial_{\bm{n}}\phi_n\|_{L^2(\varGamma_1)}^2(k_i+\lambda_n)^{-2}}{N^{1-1.5/d}}\\
= & \lim_{N\to\infty} \frac{\|\partial_{\bm{n}}\phi_{N+1}\|_{L^2(\varGamma_1)}^2(k_i+\lambda_{N+1})^{-2}}{-(N+1)^{1-1.5/d}+N^{1-1.5/d}}=0,
\end{split}
\end{equation*}
which leads to the estimate
\begin{equation}\label{est-12}
\zeta_{i,j}\cdot\|\partial_{\bm{n}}\phi_j\|_{L^2(\varGamma_1)}^{-2} \lesssim N^{1-1.5/d}.
\end{equation}
Combining \eqref{est-6}, \eqref{est-11}, \eqref{est-12} with Proposition \ref{pro-1}, for $d\in\{1,2,3\}$, we conclude that  \eqref{condition-1} is always satisfied when selecting  sufficiently large $N$.\end{proof} 
\begin{remark}
Theorem \ref{stability-thm} demonstrates that the linear component of the closed-loop system \eqref{evolution} achieves quantitative rapid stabilization. This property is fundamental for establishing the stabilization of the original nonlinear system \eqref{evolution}.
\end{remark}
\begin{lemma}\label{lem2}
 For a sufficiently large integer $N$ and spatial dimensions $d \in \{1, 2, 3\}$, the following estimate holds:
\begin{equation}\label{z-estimate}
\|z(\cdot,t)\| \le 2\|p(\cdot,t)\|.
\end{equation}
\end{lemma}
\begin{proof}
Recalling \eqref{dynamic-extension}, we have
\begin{equation}\label{est-10}
\|z\|\le \|p\|+\bigg\|\sum_{i=1}^N \mathscr{D}_iU_i\bigg\|.
\end{equation}
It follows the Cauchy--Schwarz inequality that
\[
\bigg\|\sum_{i=1}^N \mathscr{D}_iU_i\bigg\|^2
\le N\sum_{i=1}^N \|\mathscr{D}_iU_i\|^2 .
\]
Moreover, the Dirichlet operator estimate in \cite{munteanu2017stabilisation} yields the existence of a constant
$\gamma>0$ (independent of $i$ and $N$) such that
\[
\|\mathscr{D}_iU_i\|^2\le \gamma\,\|U_i\|_{L^2(\varGamma_1)}^2,\; i=1,\dots,N.
\]
Recalling the definition of $U_i$ in \eqref{Ui} and $\bm{B}$ in \eqref{a} gives
\[
\sum_{i=1}^N \|U_i\|_{L^2(\varGamma_1)}^2 = \bm{u}^\top \bm{B}\bm{u}
\le \|\bm{B}\|_2\,\|\bm{K}\|_2^2\|\bm{p}_s\|_{\mathbb{R}^N}^2 .
\]
Similar to \eqref{est-7}, we obtain
\begin{equation}\label{b}
    \|\bm{B}\|_2 \le \lambda_N^{\frac32}(1+o(1)).
\end{equation}
Meanwhile, based on \eqref{est-6}, it can be deduced that
\[
\bigg\|\sum_{i=1}^N \mathscr{D}_iU_i\bigg\|^2 \lesssim N\lambda_N^{-\frac32}\|p\|^2 .
\]
Therefore, for sufficiently large $N$, \eqref{z-estimate} holds.
\end{proof}

Based on Lemma \ref{lem2}, we  list the following lemma which is used to illustrate  the property of the nonlinear operator $\mathcal{G}(p)$.
\begin{lemma}\label{lem3}
Consider $\mathcal{G}(p)$ defined in \eqref{ag}. For a sufficiently large integer $N$ and spatial dimensions $d \in \{1, 2, 3\}$, the following estimate holds:
\begin{equation}\label{estimate-9}
\|\mathcal{G}(p)\|\leq 6\sqrt{N}L\|p(\cdot,t)\|.
\end{equation}
\end{lemma}
\begin{proof}
As preceding in Lemma \ref{lem2}, we derive
\begin{equation*}
    \left\| \sum_{i=1}^{N} \mathscr{D}_i ( \bm{T}_i^{\top} \bm{K} (\bm{I} + \bm{B}\bm{K})^{-1} \bm{f}_s ) \right\|^2 \le \gamma N \| \bm{B} \|_2 \| \bm{K} \|_2^2 \left\| (\bm{I} + \bm{B}\bm{K})^{-1} \bm{f}_s \right\|^2 .
\end{equation*}
By invoking \eqref{est-6} and \eqref{b}, we derive  
\begin{equation}\label{est-1}
\begin{split}
    \bigg\|\sum_{i=1}^{N} \mathscr{D}_i (\bm{T}_i^\top \bm{K} (\bm{I}+\bm{B}\bm{K})^{-1} \bm{f}_s)\bigg\|^2 
        \leq \gamma N\lambda_N^{-\frac{3}{2}}\|(\bm{I}+\bm{B}\bm{K})^{-1}\bm{f}_s\|^2(1+o(1)).
\end{split}
\end{equation}
 Utilizing the factorization $(\bm{I}+\bm{B}\bm{K})^{-1} = \bm{K}^{-1}(\bm{K}^{-1}+\bm{B})^{-1}$, a direct computation yields
\begin{equation*}
\|(\bm{I}+\bm{B}\bm{K})^{-1}\|_2 \leq \|(\bm{K}^{-1}+\bm{B})^{-1}\|_2 \|\bm{K}^{-1}\|_2.
\end{equation*}
By following a derivation analogous to  \eqref{est-7}, we obtain
\begin{equation}\label{K-B}
\begin{split}
\sigma_{\min}(\bm{K}^{-1}+\bm{B}) \geq \min_{1\leq j \leq N} \left\{ \frac{\|\partial_{\bm{n}}\phi_j\|_{L^2(\varGamma_1)}}{m-\lambda_j} \right\} \lambda_N^{\frac{3}{4}}(1+o(1)).
\end{split}
\end{equation}
Consequently, combining \eqref{K-B} with \eqref{est-7} gives
\begin{equation}\label{est-13}
\|(\bm{I}+\bm{B}\bm{K})^{-1}\|_2 \leq  \lambda_N^{\frac{3}{4}}(1+o(1)).
\end{equation}
In view of the Lipschitz property and Lemma \ref{lem2}, it follows that
\begin{equation}\label{est-2}
\begin{split}
        \|\bm{f}_s\|_{\mathbb{R}^N}\leq \|f\|\leq 2L\|z(\cdot,t)\|.
\end{split}
\end{equation}
Substituting \eqref{est-13},\eqref{est-2} into \eqref{est-1}, for sufficiently large $N$, we obtain
\begin{equation}\label{est-4}
\begin{split}
     \bigg\|\sum_{i=1}^{N} \mathscr{D}_i (\bm{T}_i^\top \bm{K} (\bm{I}+\bm{B}\bm{K})^{-1} \bm{f}_s)\bigg\|
     \leq& 4 \sqrt{N}L \|p(\cdot,t)\|.
\end{split}
\end{equation}
Finally, by recalling the definition of $\mathcal{G}(p)$ given below \eqref{evolution} and invoking \eqref{est-2}, \eqref{est-4}, we arrive at \eqref{estimate-9}.
\end{proof}

Based on the Theorem \ref{stability-thm}, Lemma \ref{lem2} and Lemma \ref{lem3}, we establish the following theorem.
\begin{theorem}\label{state-feedback-thm}
Consider the system \eqref{new} under the control law \eqref{Ui}. For any initial condition $p^0 \in L^2(\Omega)$ and spatial dimensional $d\in\{1,2,3\}$, there exist a sufficiently large integer $N \in \mathbb{N}^+$ and a tuning parameter $m>6\sqrt{N}L+0.5$ such that the system \eqref{new} admits a unique solution $p(\cdot,t)$ fulfilling the exponential decay estimate:
\begin{equation}\label{++z2}
\|p(\cdot,t)\|^2 \leq \kappa_2 e^{-\kappa_3 t} \|p(\cdot,0)\|^2,
\end{equation}
where $\kappa_2, \kappa_3 > 0$ are constants.
\end{theorem}
\begin{proof}
We first address the well-posedness of the closed-loop system \eqref{evolution}. Since the operator $\tilde{\mathcal{A}}$ generates an analytic semigroup on $L^2(\Omega)$, it follows from \cite[Theorem 6.3.3]{pazy2012semigroups} that for any initial condition $p^0 \in L^2(\Omega)$, the system \eqref{evolution} admits a unique classical solution $p \in C([0, \infty); L^2(\Omega))$. 
 
Next, we analyze the stability of the closed-loop system \eqref{evolution}. According to the Duhamel formula, the solution to \eqref{evolution} can be expressed as 
\begin{equation}\label{banqun}
    p(t)=e^{\tilde{\mathcal{A}}t} p^0 + \int_0^t e^{\tilde{\mathcal{A}}(t-s)}\,\mathcal{G}(p(s))\,ds.
\end{equation}
By taking the $L^2$-norm of both sides in \eqref{banqun} and invoking Lemma \ref{lem3}, we obtain 
\begin{equation*}\label{estimate-10}
\begin{split}
        \|p(t)\|
&\le e^{-(m-0.5) t}\|p^0\|
+ 6\sqrt{N}L \int_0^t e^{-(m-0.5) (t-s)} \|p(s)\|\,ds . 
\end{split}
\end{equation*}
We define the weighted function 
$$y(t):=e^{(m-0.5) t}\|p(t)\|.$$
Multiplying the inequality by $e^{(m-0.5) t}$ yields 
$$y(t)\le \|p^0\| + 6\sqrt{N}L \int_0^t y(s)\,ds.$$ 
By applying Gronwall's inequality, we have 
$$y(t)\le \|p^0\| e^{6\sqrt{N}L t}.$$ 
Consequently, we arrive at 
$$\|p(t)\|\le e^{-(m-0.5-6\sqrt{N}L)t}\|p^0\|.$$ 
This completes the proof of \eqref{++z2}.
\end{proof}
     
\subsection{Output feedback control} \label{sec3.3}
In this subsection, we construct the observer-based output feedback control law as follow:
\begin{equation}\label{observer-based control}
    \hat{\bm{u}}=\bm{K}\hat{\bm{p}}_s=\bm{K}\bm{p}_s - \bm{K}\bm{\epsilon}_s,
\end{equation}
where $\hat{\bm{p}}_s=[\hat{p}_1,\dots,\hat{p}_{N}]^{\top}$ and $\bm{\epsilon}_s=[\epsilon_1,\dots,\epsilon_N]^{\top}$. 

To validate this design, we rely on the following established properties: Theorem \ref{errorthm} establishes the exponential stability of the estimation error $\epsilon$. Furthermore, Theorem \ref{state-feedback-thm} guarantees the exponential stabilization of the system under the state feedback control law designed in \eqref{U}.  In this subsection, we leverage the results established in Theorem \ref{errorthm} and Theorem \ref{state-feedback-thm} to obtain that the proposed observer-based control effectively stabilizes the system.

Substituting the control law \eqref{observer-based control} into the system \eqref{new} yields the following closed-loop dynamics:
\begin{equation*}
\begin{cases}
    \partial_t p = \Delta p + f\bigg(p+\sum\limits_{i=1}^{N}\mathscr{D}_i\hat{U}_i\bigg) - \sum\limits_{i=1}^{N} \big[\mathscr{D}_i (\dot{\hat{U}}_i-k_i\hat{U}_i)-2\lambda_i\langle \mathscr{D}_i\hat{U}_i,\phi_i\rangle\phi_i\big], \\
    \hspace{9.5cm}x\in \Omega, \\
    p = 0, \; x\in \partial\Omega, \\
    p(x,0) = z^0(x), \;x\in \Omega.
\end{cases}
\end{equation*}
   The closed-loop system can be reformulated as the following abstract evolution equation:
\begin{equation*}
    \dot{p}(t) = \mathcal{A}_{cl} p(t)  + \mathcal{R}(\epsilon),
\end{equation*}
where $\mathcal{A}_{cl}$ and perturbation term $\mathcal{R}(\epsilon)$ are defined as:
\begin{equation*}
\begin{aligned}
    \mathcal{A}_{cl}p &= \Delta p + f\bigg(p + \sum_{i=1}^N \mathscr{D}_i U_i\bigg)   - \sum\limits_{i=1}^{N} \big[\mathscr{D}_i (\dot{U}_i-k_iU_i)-2\lambda_i\langle \mathscr{D}_iU_i,\phi_i\rangle\phi_i\big], \\
    \mathcal{R}(\epsilon) &= \left[ f\bigg(p + \sum_{i=1}^N \mathscr{D}_i (U_i + U_i^\epsilon)\bigg) - f\bigg(p + \sum_{i=1}^N \mathscr{D}_i U_i\bigg) \right] - \sum_{i=1}^N \mathscr{D}_i (\dot{U}_i^\epsilon + k_i U_i^\epsilon).
\end{aligned}
\end{equation*}
Here, the operator $\mathcal{A}_{cl}$ generates the  semigroup $S_{cl}(t)$. With this formulation established, we present the main theorem of this section.
\begin{theorem}\label{thm3}
Consider the closed-loop system consisting of the plant \eqref{new}, the observer \eqref{observer-1}, and the observer-based  control  \eqref{observer-based control}. Suppose that the state observer \eqref{observer-1} is designed with a sensor configuration satisfying condition \eqref{condition6}. For any initial condition $p^0 \in L^2(\Omega)$ and and spatial dimensional $d\in\{1,2,3\}$, there exist a tuning parameter $m>6\sqrt{N}L+0.5$ and a sufficiently large integer $N \in \mathbb{N}^+$ such that the closed-loop system admits a unique classical solution $p(\cdot,t)$. Furthermore, the following exponential decay estimates hold:
\begin{equation*}\label{z3}
    \|p(\cdot,t)\|^2 \leq \kappa_4 e^{-\kappa_5 t} \|p^0\|^2,
\end{equation*}
and the system $z(x,t)$ defined in \eqref{old} satisfies
\begin{equation*}
    \|z(\cdot,t)\|^2 \leq \kappa_6 e^{-\kappa_5 t} \|z^0\|^2,
\end{equation*}
where $\kappa_i > 0$ for $i = 4, 5, 6$ are positive constants.
\end{theorem}
\begin{proof}
    As a consequence of Theorem \ref{errorthm}, there exist constants $\kappa_7, \kappa_8 > 0$ such that:
    \begin{equation}\label{eq:perturbation_bound}
         \left\| \sum_{i=1}^N \mathscr{D}_i (\dot{U}_i^\varepsilon + k_i U_i^\varepsilon) \right\| \le \kappa_7 e^{-\kappa_8 t}.
    \end{equation}
    Combining the global Lipschitz continuity of $f(\cdot)$ with  \eqref{eq:perturbation_bound}, we get
    \begin{equation*}
        \|\mathcal{R}(\epsilon)\| \le L \left\| \sum_{i=1}^N \mathscr{D}_i U_i^\varepsilon \right\| + \left\| \sum_{i=1}^N \mathscr{D}_i (\dot{U}_i^\varepsilon + k_i U_i^\varepsilon) \right\| \le \kappa_9 e^{-\kappa_8 t}.
    \end{equation*}
    The mild solution of the system is given by Duhamel's formula:
    \begin{equation*}
        p(t) = S_{cl}(t)p(0) + \int_0^t S_{cl}(t-\tau) \mathcal{R}(\tau) d\tau.
    \end{equation*}
    Taking norms and applying the semigroup bound $\|S_{cl}(t)\| \le \kappa_2 e^{-\kappa_3 t}$ yields:
    \begin{equation*}
        \|p(t)\| \le \kappa_2 e^{-\mu t} \|p(0)\| + \kappa_9\int_0^t  e^{-\kappa_3 (t-\tau)}  e^{-\kappa_8 \tau} d\tau.
    \end{equation*}
   The proof is completed.
\end{proof}
\section{Illustrative Example}\label{sec4}
The objective of the numerical experiment is twofold: (1) to validate the feasibility of the stability condition, and (2) to illustrate the effectiveness of the proposed output feedback control and observer design.
\subsection{Numerical verification of some conditions}\label{sec4.1}
 We consider the 2-D nonlinear parabolic system on $\Omega = (0, 1)\times (0,1)$ with $\varGamma_1 = \{0\} \times (0, 1)$, $h(x,y,t)=z(0.5,y,t)$, $f(z)=50\sin(z)+50z$, and $z^0=\cos(x)$. To verify the feasibility of \eqref{condition-1}, we revisit the key terms in Lemma \ref{lem3} and Theorem \ref{stability-thm}. These terms exhibit that
\begin{equation}\label{3-term}
\begin{split}
    &\|\bm{K}\|_2 \lesssim N^{-\frac{3}{2}}, \quad \|(\bm{I}+\bm{B}\bm{K})^{-1}\|_2 \lesssim N^{\frac{3}{4}}, \quad \sum_{i,j=1}^N \frac{\zeta_{i,j}}{(k_i-\lambda_j)^{2}} \lesssim  N^{\frac{11}{4}}.
\end{split}
\end{equation}
\begin{figure}[htbp]
    \centering
    \includegraphics[width=1\textwidth]{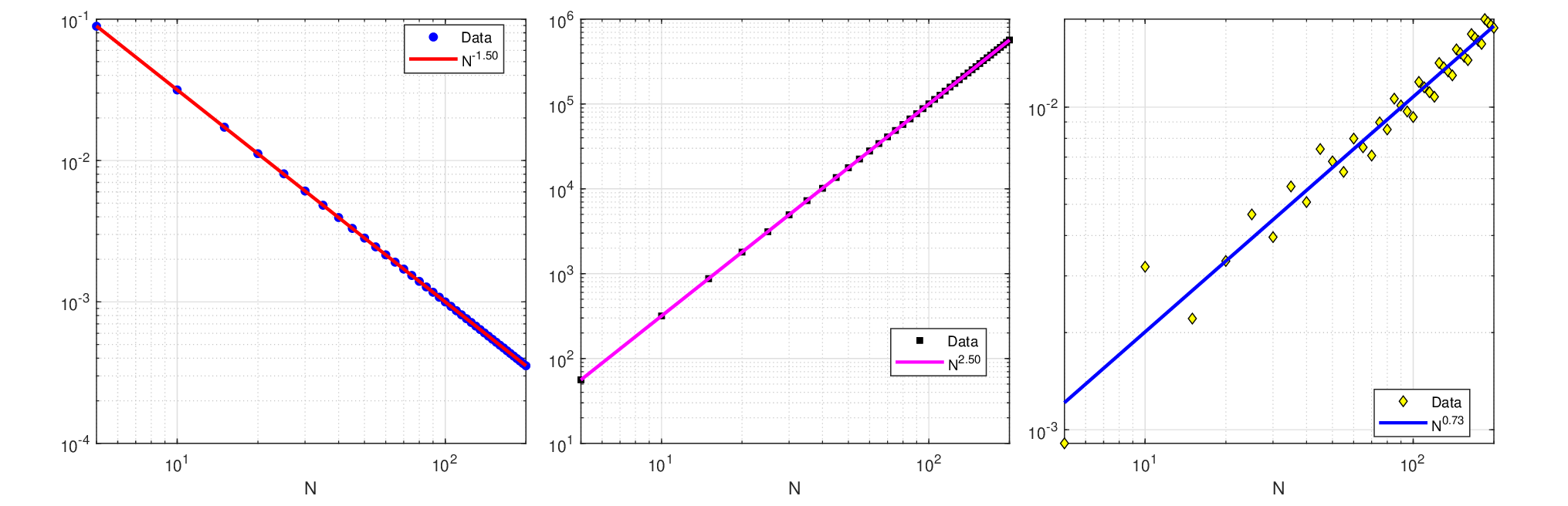}
    \caption{Numerical verification of the asymptotic scaling for \eqref{3-term}
    }
    \label{figure3}
\end{figure}
Fig. \ref{figure3} displays the asymptotic scaling of the key terms in \eqref{3-term} on a log-log scale for $N \in [5, 200]$. The subplots, from left to right, illustrate that the numerical data (markers) for $\|\bm{K}\|_2$, $\sum_{i,j=1}^N (k_i-\lambda_j)^{-2}\zeta_{i,j}$, and $\|(\bm{I}+\bm{B}\bm{K})^{-1}\|_2$ scale as $N^{-1.50}$, $N^{2.50}$, and $N^{0.73}$ (solid lines), respectively. This high consistency confirms the feasibility of the proposed control conditions for sufficiently large $N$. 

\subsection{Numerical verification of controller and observer}\label{sec4.2}
Numerical simulations are performed using the Galerkin method ($M_{modes} = 120$) and MATLAB's \texttt{ode15s} to validate the proposed strategy. With $L=100$, $N=6$, and $m=120$, Fig. \ref{fig:combined}(a) and (b) illustrate that the closed-loop system achieves rapid exponential stabilization, whereas the open-loop response is divergent. Furthermore, the exponential decay of the observer estimation error is confirmed in Fig. \ref{fig:combined}(c), which validates the theoretical error dynamics.
\begin{figure}[htbp]
    \centering
    \begin{minipage}{0.4\columnwidth}
        \centering
        \includegraphics[width=\textwidth]{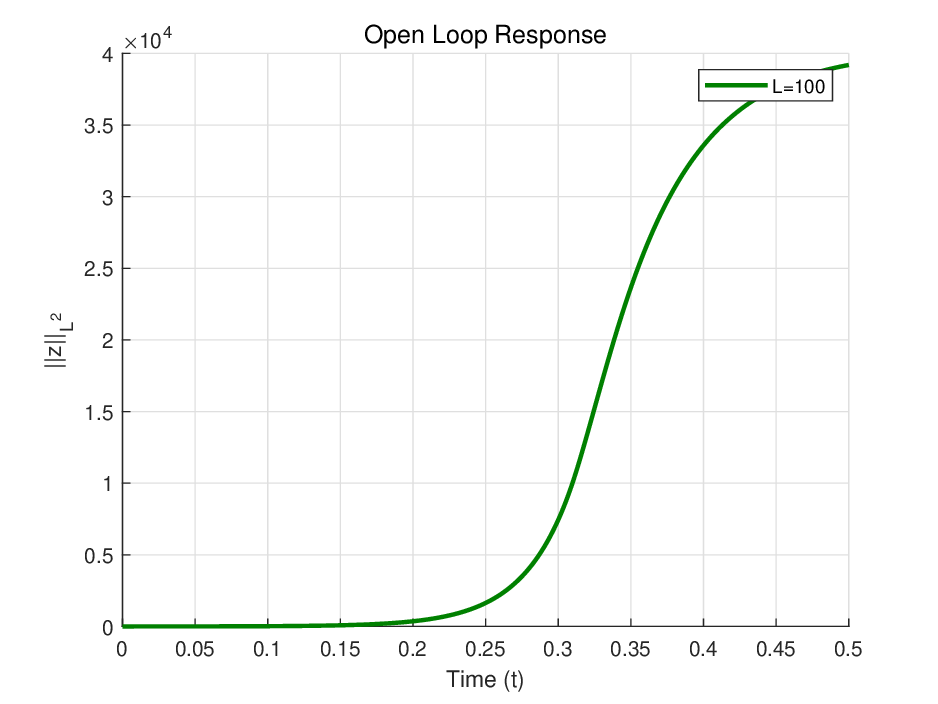}
        \vspace{-2pt} 
        \centerline{\footnotesize (a)}
    \end{minipage}
    \hspace{-2pt} 
    \begin{minipage}{0.4\columnwidth}
        \centering
        \includegraphics[width=\textwidth]{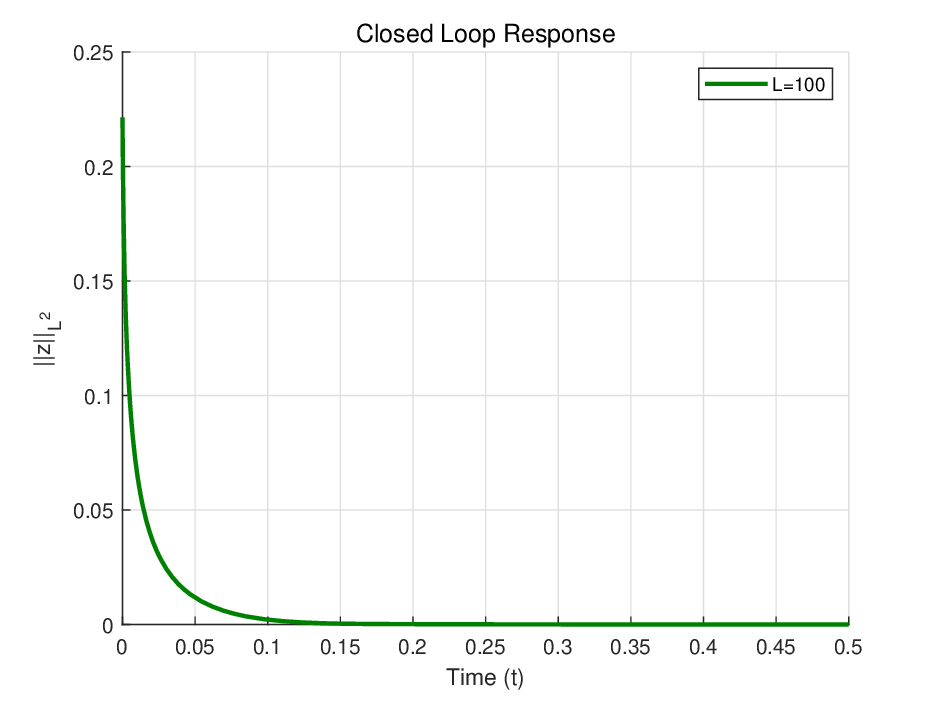}
        \vspace{-2pt}
        \centerline{\footnotesize (b)}
    \end{minipage}
    \hspace{-2pt} 
    \begin{minipage}{0.4\columnwidth}
        \centering
        \includegraphics[width=\textwidth]{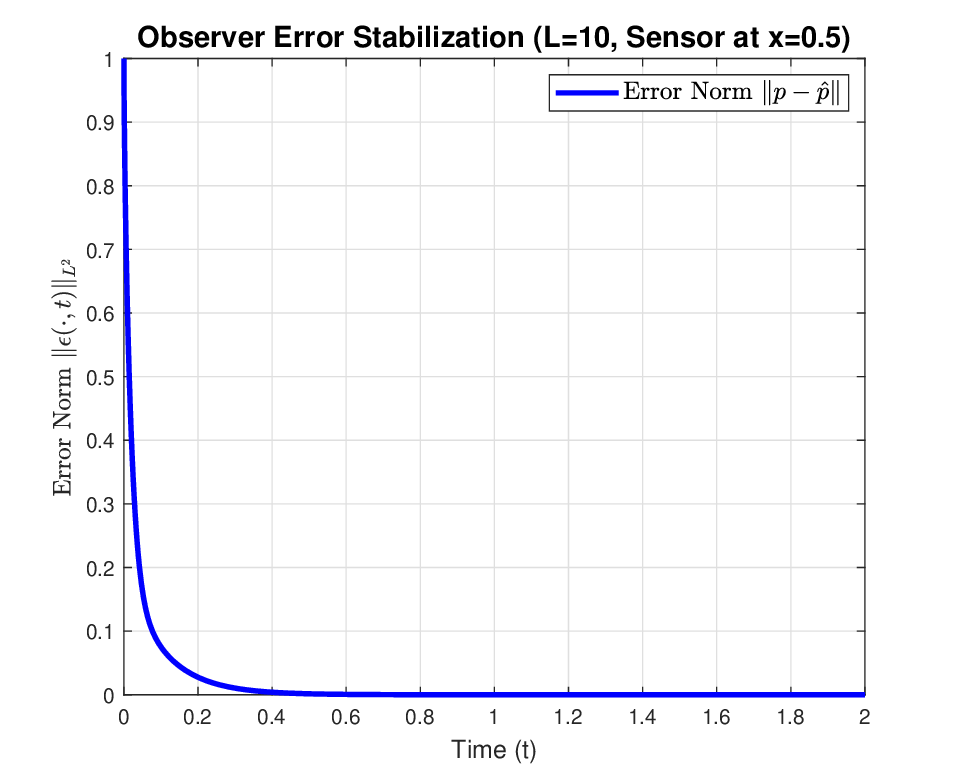}
        \vspace{-2pt}
        \centerline{\footnotesize (c)}
    \end{minipage}

    \vspace{2pt}
    {\centering \caption{Numerical results for $L=100$: (a) open-loop, (b) closed-loop, (c) estimation error.}\label{fig:combined}\par}
\end{figure}

	\section{Conclusion}\label{sec5}
  This paper addresses the output feedback stabilization of semilinear parabolic PDEs in general multidimensional domains. First, we develop a novel nonlinear observer that ensures exponential state tracking. For any given Lipschitz constant, we show that the sensor configuration can be strategically selected via the Berezin–Li–Yau inequality to maintain estimation performance. To the best of the authors' knowledge, this is the first result on nonlinear observer design for such systems over general multidimensional domains. Building upon this framework, a finite-dimensional observer-based control law is synthesized to guarantee exponential stabilization. The feasibility of the design is rigorously verified for spatial dimensions $d \in \{1, 2, 3\}$, overcoming the constraints on Lipschitz constant magnitude prevalent in previous literature. Numerical simulations confirm the efficacy of the proposed strategy.

  Two promising directions for future research are identified: 
  \begin{itemize}
      \item[(1)] Relaxing the Lipschitz continuity requirement to accommodate more general nonlinearities, though this may necessitate constraints on the initial conditions, as seen in linear control studies \cite{barbu2013boundary,munteanu2017stabilisation}.
      \item[(2)]  Extending the framework to stochastic multidimensional parabolic equations under multiplicative noise to broaden its applicability.
  \end{itemize} 

	\section*{Funding}
	This work was supported by Natural Science Foundation of China NSFC-62473281, 62173348, 12571484.


\begin{thebibliography}{30}
\providecommand{\natexlab}[1]{#1}
\providecommand{\url}[1]{\texttt{#1}}
\expandafter\ifx\csname urlstyle\endcsname\relax
  \providecommand{\doi}[1]{doi: #1}\else
  \providecommand{\doi}{doi: \begingroup \urlstyle{rm}\Url}\fi

\bibitem[Barbu(2013)]{barbu2013boundary}
V.~Barbu.
\newblock Boundary stabilization of euilibrium solutions to parabolic
  equations.
\newblock \emph{IEEE Transactions on Automatic Control}, 58\penalty0
  (9):\penalty0 2416--2420, 2013.

\bibitem[Berezin(1972)]{berezin1972convex}
F.~A. Berezin.
\newblock Convex operator functions.
\newblock \emph{Mathematics of the USSR-Sbornik}, 17\penalty0 (2):\penalty0
  269, 1972.

\bibitem[Chowdhury et~al.(2024)Chowdhury, Dutta, and
  Majumdar]{chowdhury2024local}
S.~Chowdhury, R.~Dutta, and S.~Majumdar.
\newblock Local exponential stabilization of {R}ogers--{M}cculloch and
  {F}itzhugh--{N}agumo equations by the method of backstepping.
\newblock \emph{ESAIM: Control, Optimisation and Calculus of Variations},
  30:\penalty0 41, 2024.

\bibitem[El~Harraki and Lamrani(2023)]{el2023exponential}
I.~El~Harraki and I.~Lamrani.
\newblock Exponential stabilization of semilinear parabolic systems.
\newblock \emph{Journal of the Franklin Institute}, 360\penalty0 (11):\penalty0
  7294--7324, 2023.

\bibitem[Evans(2010)]{evans2010partial}
L.~C. Evans.
\newblock \emph{Partial Differential Equations}.
\newblock American Mathematical Society, Providence, Rhode Island, 2010.

\bibitem[Feng et~al.(2022)Feng, Lang, and Liu]{feng2022boundary}
H.~Y.~P. Feng, P.~H. Lang, and J.~K. Liu.
\newblock Boundary stabilization and observation of a weak unstable heat
  equation in a general multi-dimensional domain.
\newblock \emph{Automatica}, 138:\penalty0 110152, 2022.

\bibitem[Furter and Grinfeld(1989)]{furter1989local}
J.~Furter and M.~Grinfeld.
\newblock Local vs. non-local interactions in population dynamics.
\newblock \emph{Journal of Mathematical Biology}, 27:\penalty0 65--80, 1989.

\bibitem[Hagen and Mezic(2003)]{hagen2003spillover}
G.~Hagen and I.~Mezic.
\newblock Spillover stabilization in finite-dimensional control and observer
  design for dissipative evolution equations.
\newblock \emph{SIAM Journal on Control and Optimization}, 42\penalty0
  (2):\penalty0 746--768, 2003.

\bibitem[Hagen et~al.(2004)Hagen, Mezic, and Bamieh]{hagen2004distributed}
G.~Hagen, I.~Mezic, and B.~Bamieh.
\newblock Distributed control design for parabolic evolution equations:
  application to compressor stall control.
\newblock \emph{IEEE Transactions on Automatic Control}, 49\penalty0
  (8):\penalty0 1247--1258, 2004.

\bibitem[Hassell and Tao(2002)]{hassell2002upper}
A.~Hassell and T.~Tao.
\newblock Upper and lower bounds for normal derivatives of {D}irichlet
  eigenfunctions.
\newblock \emph{Mathematical Research Letters}, 9:\penalty0 289--305, 2002.

\bibitem[Kaczor and Nowak(2000)]{kaczor2000problems}
W.~J. Kaczor and M.~T. Nowak.
\newblock \emph{Problems in Mathematical Analysis I: Real Numbers, Sequences
  and Series}.
\newblock American Mathematical Society, Providence, Rhode Island, 2000.

\bibitem[Karafyllis(2021)]{karafyllis2021lyapunov}
I.~Karafyllis.
\newblock Lyapunov-based boundary feedback design for parabolic {PDEs}.
\newblock \emph{International Journal of Control}, 94\penalty0 (5):\penalty0
  1247--1260, 2021.

\bibitem[Katz and Fridman(2022)]{katz2022global}
R.~Katz and E.~Fridman.
\newblock Global finite-dimensional observer-based stabilization of a
  semilinear heat equation with large input delay.
\newblock \emph{Systems \& Control Letters}, 165:\penalty0 105275, 2022.

\bibitem[Katz and Fridman(2023)]{katz2023global}
R.~Katz and E.~Fridman.
\newblock Global stabilization of a {1D} semilinear heat equation via modal
  decomposition and direct {Lyapunov} approach.
\newblock \emph{Automatica}, 149:\penalty0 110809, 2023.

\bibitem[Lasiecka and Triggiani(2000)]{Lasiecka_Triggiani_2000}
I.~Lasiecka and R.~Triggiani.
\newblock \emph{Control Theory for Partial Differential Equations: Continuous
  and Approximation Theories}.
\newblock Cambridge University Press, 2000.

\bibitem[Lhachemi and Prieur(2022{\natexlab{a}})]{lhachemi2022global}
H.~Lhachemi and C.~Prieur.
\newblock Global output feedback stabilization of semilinear reaction-diffusion
  pdes.
\newblock \emph{IFAC-PapersOnLine}, 55\penalty0 (26):\penalty0 53--58,
  2022{\natexlab{a}}.

\bibitem[Lhachemi and Prieur(2022{\natexlab{b}})]{lhachemi2022nonlinear}
H.~Lhachemi and C.~Prieur.
\newblock Nonlinear boundary output feedback stabilization of
  reaction--diffusion equations.
\newblock \emph{Systems \& Control Letters}, 166:\penalty0 105301,
  2022{\natexlab{b}}.

\bibitem[Lhachemi et~al.(2025)Lhachemi, Munteanu, and
  Prieur]{lhachemi2025boundary}
H.~Lhachemi, I.~Munteanu, and C.~Prieur.
\newblock Boundary output feedback stabilization for {2-D} and {3-D} parabolic
  equations.
\newblock \emph{Automatica}, 176:\penalty0 112259, 2025.

\bibitem[Li and Yau(1983)]{li1983schrodinger}
P.~Li and S.~T. Yau.
\newblock On the schr{\"o}dinger equation and the eigenvalue problem.
\newblock \emph{Communications in Mathematical Physics}, 88\penalty0
  (3):\penalty0 309--318, 1983.

\bibitem[Liu et~al.(2025)Liu, Han, and Peng]{liu2025output}
K.~Liu, Z.~J. Han, and X.~Y. Peng.
\newblock Output feedback control for non-linear stochastic reaction--diffusion
  system via modal decomposition techniques.
\newblock \emph{IMA Journal of Mathematical Control and Information},
  42\penalty0 (2):\penalty0 dnaf011, 2025.

\bibitem[Meurer(2013)]{meurer2013extended}
T.~Meurer.
\newblock On the extended {Luenberger}-type observer for semilinear
  distributed-parameter systems.
\newblock \emph{IEEE Transactions on Automatic Control}, 58\penalty0
  (7):\penalty0 1732--1743, 2013.

\bibitem[Munteanu(2017)]{munteanu2017stabilisation}
I.~Munteanu.
\newblock Stabilisation of parabolic semilinear equations.
\newblock \emph{International Journal of Control}, 90\penalty0 (5):\penalty0
  1063--1076, 2017.

\bibitem[Pazy(2012)]{pazy2012semigroups}
A.~Pazy.
\newblock \emph{Semigroups of Linear Operators and Applications to Partial
  Differential Equations}.
\newblock Springer Science \& Business Media, New York, 2012.

\bibitem[Qi(1984)]{qi1984some}
L.~Qi.
\newblock Some simple estimates for singular values of a matrix.
\newblock \emph{Linear algebra and its applications}, 56:\penalty0 105--119,
  1984.

\bibitem[Smyshlyaev and Krstic(2005)]{smyshlyaev2005backstepping}
A.~Smyshlyaev and M.~Krstic.
\newblock Backstepping observers for a class of parabolic {PDEs}.
\newblock \emph{Systems \& Control Letters}, 54\penalty0 (7):\penalty0
  613--625, 2005.

\bibitem[Temam(2012)]{temam2012infinite}
R.~Temam.
\newblock \emph{Infinite-dimensional Dynamical Systems in Mechanics and
  Physics}.
\newblock Springer Science \& Business Media, 2012.

\bibitem[Vazquez and Krstic(2008)]{vazquez2008control}
R.~Vazquez and M.~Krstic.
\newblock Control of {1D} parabolic {PDEs} with {Volterra} nonlinearities, part
  {II}: analysis.
\newblock \emph{Automatica}, 44\penalty0 (11):\penalty0 2791--2803, 2008.

\bibitem[Wang and Fridman(2024)]{wang2024delayed}
P.~F. Wang and E.~Fridman.
\newblock Delayed finite-dimensional observer-based control of {2D} linear
  parabolic pdes.
\newblock \emph{Automatica}, 164:\penalty0 111607, 2024.

\bibitem[Weyl(1912)]{weyl1912asymptotische}
H.~Weyl.
\newblock Das asymptotische verteilungsgesetz der eigenwerte linearer
  partieller differentialgleichungen (mit einer anwendung auf die theorie der
  hohlraumstrahlung).
\newblock \emph{Mathematische Annalen}, 71\penalty0 (4):\penalty0 441--479,
  1912.

\bibitem[Xiang(2024)]{xiang2024}
S.~Q. Xiang.
\newblock Quantitative rapid and finite time stabilization of the heat
  equation.
\newblock \emph{ESAIM: Control, Optimisation and Calculus of Variations},
  30\penalty0 (40):\penalty0 1--25, 2024.

\end{thebibliography}
\end{document}